\documentclass[11pt]{amsart}
\usepackage{amssymb}
\usepackage{multirow}
\usepackage{mathtools}
\usepackage{hyperref}
\hypersetup{breaklinks=true}
\usepackage{breakurl}
\usepackage{enumitem}
\usepackage{tikz-cd}

\usepackage{color}

\usepackage[margin=1in]{geometry}

\newtheorem{thm}{Theorem}[section]
\newtheorem{cor}[thm]{Corollary}
\newtheorem{prop}[thm]{Proposition}
\newtheorem{lem}[thm]{Lemma}

\theoremstyle{definition}
\newtheorem{defn}[thm]{Definition}

\theoremstyle{remark}

\newtheoremstyle{TheoremNum}
        {7pt}{7pt}              
        {\itshape}                      
        {}                              
        {\bfseries}                     
        {.}                             
        { }                             
        {\thmname{#1}\thmnote{ \bfseries #3}}
    \theoremstyle{TheoremNum}
    \newtheorem{thmn}{Theorem}

\begin{document}

\title{Some Explicit Elliptic Modular Surfaces}
\author{Laure Flapan}
\address{Department of Mathematics, Northeastern University, USA}
\email{l.flapan@northeastern.edu}

\begin{abstract}
We consider algebraic surfaces, recently constructed by Schreieder, that are smooth models of the quotient of the self-product of a complex hyperelliptic curve by a $(\mathbb{Z}/3^c\mathbb{Z})$-action. We show that these surfaces are elliptic modular surfaces in the sense of Shioda, meaning in particular that they are universal families of explicit moduli of elliptic curves. 
\end{abstract}
\maketitle

\section{Introduction}
For any $c\ge1$ we consider a complex hyperelliptic curve $C_g$ of genus $g=\frac{3^c-1}{2}$, equipped with a $(\mathbb{Z}/3^c\mathbb{Z})$-action. By equipping the product $C_g\times C_g$ with a $(\mathbb{Z}/3^c\mathbb{Z})$-action,  Schreieder \cite{schreieder} constructs a smooth complex projective surface $X_c$ as a smooth model of the quotient of $C_g\times C_g$ by this $(\mathbb{Z}/3^c\mathbb{Z})$-action. Schreieder constructed these surfaces as part of a larger class smooth models $Z_{c,n}$ of the quotient of $C_g^n$ by a group isomorphic to $(\mathbb{Z}/3^c\mathbb{Z})^{n-1}$ for $n\ge 2$. These $n$-dimensional varieties $Z_{c,n}$ are notable for the unexpected form of their Hodge diamonds, namely that their Hodge numbers satisfy $h^{n,0}=g$ and $h^{p,q}=0$ for all other $p\ne q$. 

The noteworthy Hodge-theoretic properties of the varieties $Z_{c,n}$ yield noteworthy arithmetic and cycle-theoretic properties as well. For $c=1$, Cynk and Hulek prove in \cite{cynk} that the varieties $Z_{1,n}$ are modular, in the sense that one may identify a corresponding $L$-function attached to their cohomology. This result was then extended to all $c\ge 1$ in \cite{FL}. Recently, Laterveer and Vial show in \cite{LV} that the subring of the Chow ring of $Z_{c,n}$ generated by divisors, Chern classes, and intersections of two positive-dimensional cycles injects into cohomology via the cycle class map. Moreover they show that in the surface case studied here, the small diagonal of $Z_{c,2}$ admits a decomposition similar to that of $K3$ surfaces proved by Beauville-Voisin \cite{BV}. 

Here we investigate the geometry of the surfaces $Z_{c,2}=X_c$.  Shioda and Inose studied the geometry of these surfaces in the $c=1$ case in \cite{singularK3} and showed they are elliptically fibered $K3$ surfaces with finite Mordell-Weil group. In this paper, we show that for $c\ge 2$, the surface $X_c$ is similarly elliptically fibered. Moreover, we show that it is in fact an extremal elliptic surface, meaning that  $X_c$ has maximal Picard number and finite Mordell-Weil group. This adds to recent examples of Picard-maximal surfaces given by Beauville \cite{beauville} and Arapura-Solapurkar \cite{arapura}.

One of the principal constructions of extremal elliptic surfaces is due to Shioda \cite{modsur}, who attaches to any finite index subgroup $\Gamma$ of $SL(2,\mathbb{Z})$ not containing $-\mathrm{Id}$ a corresponding extremal elliptic surface $S_\Gamma$, called an \emph{elliptic modular surface}, above the modular curve $C_\Gamma \coloneqq \Gamma\backslash \mathcal{H}$ together with finitely many cusps. The elliptic modular surface $S_{\Gamma}$ is thus a universal family for the moduli space of elliptic curves parametrized by the curve $C_\Gamma$. 

Although Shioda's construction provides an appealing moduli interpretation of the elliptic surface $S_{\Gamma}$, in general, especially for complicated subgroups $\Gamma \subset SL(2,\mathbb{Z})$, finding an actual geometric construction of $S_{\Gamma}$ is quite difficult. The main result of this paper is that for a particular non-congruence subgroup $\Gamma_c$ of index $6\cdot 3^c$ in $SL(2,\mathbb{Z})$ we have:
\begin{thmn}[\ref{elmod}] For $c\ge2$, the surface $X_c$ is the elliptic modular surface attached to $\Gamma_c$. \end{thmn}

The organization of the paper is as follows. In Section \ref{construction}, we present the details of the construction of the surfaces $X_c$. Section \ref{koddim2} consists of the calculation that the surfaces $X_c$ have Kodaira dimension $1$. We then show in Section \ref{basepointfreesection} that the canonical bundle $K_{X_c}$ is basepoint free and use this in Section \ref{ellipticfibrationsection} to study the induced elliptic fibration $f\colon X_c\rightarrow \mathbb{P}^1$. The geometry of this fibration allows one to show in Section \ref{S6} that $X_c$ is  an extremal elliptic surface. Section \ref{S8} examines the $j$-invariant of $f:X_c\rightarrow \mathbb{P}^1$, which is necessary for the eventual proof in Section \ref{S9} that $X_c$ is an elliptic modular surface.

\section{Construction of $X_c$}\label{construction}

We begin by presenting the details of the construction of the surface $X_c$. For a fixed $c\ge2$ consider the complex hyperelliptic curve $C_g$ of genus $g=\frac{3^c-1}{2}$ given by the smooth projective model of the affine curve 
\[\{y^2=x^{2g+1}+1\}.\]
obtained by adding a point at $\infty$ covered by an affine piece
$\{v^2=u^{2g+2}+u\},$
such that  $x=u^{-1}$ and $y=v\cdot u^{-g-1}$. 

Fix $\zeta$ a primitive $3^c$-th root of unity. The curve $C_g$ then comes equipped with an automorphism $\psi_g$ of order $3^c=2g+1$ given by
\begin{align*}
(x,y)&\mapsto (\zeta x, y)\\
(u,v)&\mapsto (\zeta^{-1} u, \zeta^{g} v).
\end{align*}

Consider the product $C_g\times C_g$ and the automorphism of $C_g\times C_g$ given by $\psi_g^{-1}\times \psi_g$. Define the surface $X_c$ to be the minimal resolution of the quotient  $(C_g\times C_g)/\langle \psi_g^{-1}\times \psi_g \rangle$.

Observe that the fixed set of the action of $\psi_g$ on $C_g$ consists of $3$ points. In the coordinate patch given by  $\{y^2=x^{2g+1}+1\}$ we have two fixed points:
$$P_1\colon (x,y)=(0,1) \text{ and } P_2\colon (x,y)=(0,-1).$$
In the coordinate patch given by $\{v^2=u^{2g+2}+u\}$, we have a third fixed point:
$$Q\colon (u,v)=(0,0).$$

Hence, the fixed set of $\psi_g^{-1}\times \psi_g$ on $C_g\times C_g$ is $9$ points consisting of pairs of points taken from the set 
$\{P_1,P_2,Q\}$. The surface $X_c$ is then obtained from $(C_g\times C_g)/\langle \psi_g^{-1}\times \psi_g \rangle$ by resolving these $9$ singular points. 

Now using the Implicit Function Theorem one may verify that the coordinate $x$ is a local coordinate in the coordinate patch on $C_g$ given by $\{y^2=x^{2g+1}+1\}$ and the coordinate $v$ is a local coordinate in the coordinate patch on $C_g$ given by $\{v^2=u^{2g+2}+u\}$. Hence $\psi_g$ acts with weight $1$ around $P_1$ and $P_2$ and acts with weight $g$ around $Q$. 

Thus the $\mathbb{Z}/3^c\mathbb{Z}$-action of $\psi_g^{-1}\times \psi_g$ on $C_g\times C_g$ has weights $(-1,1)$ around fixed points of the form $(P_i,P_j)$, weights $(-1,g)$ around fixed points of the form $(P_i,Q)$, weights $(-g,1)$ around fixed points of the form $(Q,P_i)$, and weights $(-g,g)$ around the fixed point $(Q,Q)$.

We will call the $5$ fixed points on which $\psi_g^{-1}\times \psi_g$ has weights $(-1,1)$ or $(-g,g)$ the \emph{Type I fixed points } of $C_g\times C_g$ and the $4$ fixed points on which $\psi_g^{-1}\times \psi_g$ has weights $(-1,g)$ or $(-g,1)$ the \emph{Type II fixed points} of $C_g\times C_g$. We will also refer to these points as the Type I and Type II singular points, respectively, of the quotient $(C_g\times C_g)/\langle \psi_g^{-1}\times \psi_g \rangle$.

\subsection{Resolving the Singular Points of $(C_g\times C_g)/\langle \psi_g^{-1}\times \psi_g \rangle$}\label{quotsingsection}
To understand the resolutions of the singular points of $(C_g\times C_g)/\langle \psi_g^{-1}\times \psi_g \rangle$, we make use of established facts about surface cyclic quotient singularities and Hirzebruch-Jung resolutions.  A brief survey of these can be found in \cite[Section 2.4]{kollar} and a more detailed explanation can be found in \cite{reid}.

\subsubsection{Surface Cyclic Quotient Singularities and $X_c$}\label{cyclicquot}
Consider the action of the cyclic group $\mathbb{Z}/r\mathbb{Z}$ on $\mathbb{C}^2$ given by
$(z_1,z_2) \mapsto (\epsilon z_1, \epsilon^a z_2),$
for some $a$ coprime to $r$, where $\epsilon$ is a primitive $r$-th root of unity. We denote this action by
$\frac{1}{r}(1,a).$

For coprime integers $r$ and $a$, the \emph{Hirzebruch-Jung continued fraction} of $\frac{r}{a}$ is the expansion:
$$\frac{r}{a}=b_0-\frac{1}{b_1-\frac{1}{b_2-\frac{1}{b_3-\frac{1}{\cdots}}}}.$$
The \emph{Hirzebruch-Jung expansion} of $\frac{r}{a}$ is then the sequence
$[b_0,b_1,b_2,b_3,\ldots b_s],$ which denotes that  the minimal resolution of the singularity $\mathbb{C}^2/\frac{1}{r}(1,a)$ is a chain of $s+1$ exceptional curves $E_0, E_1,\ldots, E_s$ with nonzero intersection numbers $E_i.E_i=-b_i$ and $E_i.E_{i+1}=1$ \cite[Proposition 2.32]{kollar}.

In the quotient $(C_g\times C_g)/\langle \psi_g^{-1}\times \psi_g \rangle$, since $\psi_g^{-1}\times \psi_g$ acts on Type I fixed points with weights $(-1,1)$ on points $(P_i,P_j)$ and with weights $(-g,g)$ on the point $(Q,Q)$, the Type I fixed points have Hirzebruch-Jung expansion:
\[\underbrace{[2,\ldots ,2]}_{(3^c-1)-\mathrm{times}}.\]
In particular, the Type I singular points of $C_g\times C_g/\langle \psi_g^{-1}\times \psi_g\rangle$ are DuVal singularities of type $A_{3^c-1}$ and have minimal resolution consisting of a chain of $3^c-1$ rational curves, each with self-intersection $-2$.

Similarly, since the Type II fixed points are acted on by $\psi_g^{-1}\times \psi_g$ with weights $(-1,g)$ and $(-g,1)$, the Type II singular points  of $(C_g\times C_g)/\langle \psi_g^{-1}\times \psi_g\rangle$ have Hirzebruch-Jung expansion:
$$[2,g+1].$$
Hence the minimal resolution of each Type II singular point consists of a chain of two rational curves, one with self-intersection $-2$ and one with self-intersection $-(g+1)$.


\section{The Kodaira Dimension of $X_c$}\label{koddim2}
For a smooth algebraic variety $V$ and any $m> 0$, the $m$\emph{-th plurigenus} of $V$ is given by $P_m\coloneqq h^0(V,K_{V}^{\otimes m})$. The \emph{Kodaira dimension} $\kappa$ of $V$ is $-\infty$ if $P_m=0$ for all $m>0$ and otherwise it is the minimum $\kappa$ such that $P_m/m^\kappa$ is bounded. If $V$ has dimension $d$, then the Kodaira dimension of $V$ is either $-\infty$ or an integer $0\le \kappa \le d$.

Kodaira dimension is a birational invariant and thus we may compute the Kodaira dimension of the surface $X_c$ by computing the Kodaira dimension of any smooth model of $(C_g\times C_g)/\langle \psi_g^{-1}\times \psi_g \rangle$. In fact, we compute the Kodaira dimension of the smooth but non-minimal model $Y_c$ constructed by Schreieder in  \cite[Section 8]{schreieder} and detailed below.

\subsection{Construction of $Y_c$}\label{constructionY}
Consider the group 
$$H:=\langle \psi_g^{-1}\times \mathrm{id},\mathrm{id}\times \psi_g \rangle \subset \mathrm{Aut}(C_g\times C_g)$$
Now for each $i=1,\ldots,c$ define the element of order $3^i$ in $H$:
$$\eta_i:=(\psi_g^{-1}\times \psi_g)^{3^{c-i}}$$
The element $\eta_i$ generates a cyclic subgroup $G_i:=\langle\eta_i\rangle \subset H$, which gives a filtration
$$0=G_0\subset G_1\subset \cdots \subset G_c=\langle \psi_g^{-1}\times \psi_g\rangle$$
such that each quotient $G_i/G_{i-1}$ is cyclic of order $3$. Now,  let
\begin{equation*}
\begin{aligned}
Y_0&=C_g\times C_g,\\
Y_0'& = \text{Blow up of } Y_0 \text{ along } \mathrm{Fix}_{Y_0}(\eta_1),\\
Y_0''& =\text{Blow up of } Y_0' \text{ along } \mathrm{Fix}_{Y_0'}(\eta_1),
\end{aligned}
\end{equation*}
observing that since $H$ restricts to an action on $ \mathrm{Fix}_{Y_0}(\eta_1)$, the action of $H$ on $Y_0$ lifts to an action on $Y_0'$ and then similarly to an action on $Y_0''$. 

Now inductively define for $i\in \{1,\ldots,c\}$: 
\begin{equation*}
\begin{aligned}
Y_i&=Y_{i-1}''/\langle\eta_i\rangle,\\
Y_i'& = \text{Blow up of } Y_i \text{ along } \mathrm{Fix}_{Y_i}(\eta_{i+1}),\\
Y_i''& =\text{Blow up of } Y_i' \text{ along } \mathrm{Fix}_{Y_i'}(\eta_{i+1}),
\end{aligned}
\end{equation*}
where, by abuse of notation, we let $\langle\eta_i\rangle$ denote the subgroup of $\mathrm{Aut}(Y_{i-1}'')$ generated by the action of $\eta_i\in H$. Hence we have the following diagram:
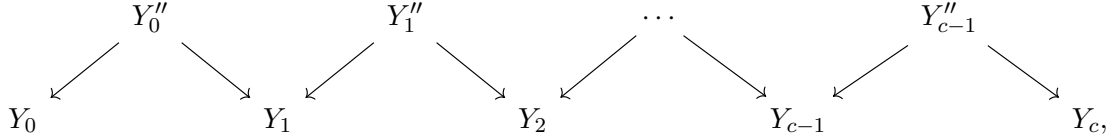
\[
\begin{tikzcd}
&\arrow{dl} Y_0''\arrow{dr}&& \arrow{dl}Y_1''\arrow{dr}&&\arrow{dl}\cdots \arrow{dr} && \arrow{dl}Y_{c-1}''\arrow{dr}&\\
Y_0 && Y_1&& Y_2&& Y_{c-1}&&Y_c,
\end{tikzcd}
\]
where each arrow to the left corresponds to a sequence of two blow-up maps and each arrow to the right corresponds to a $3:1$ cover. Each quotient $Y_{i+1}=Y_i''/\langle \eta_{i+1}\rangle$ is then a smooth model of $Y_0/G_{i+1}$. Therefore $Y_c$ is the desired smooth model of $(C_g\times C_g)/G_c$.

\subsection{Images of the Type II Singular Points}\label{typeIIsingpoints section}
For our later computations, it will be useful to understand the images of the Type II singular points of $(C_g\times C_g)/G_c$ in $Y_c$.   Consider the Type II fixed points of the form $(Q,P_i)$. We know that the automorphism $\psi_g^{-1}\times \psi_g$ acts on these fixed points with weights $(-g,1)$. Hence $\psi_g^{-1}\times \psi_g$ acts on the preimage $E_0'$ of $(Q,P_i)$ in $Y_0'$ with weights
\[(-g,g+1) \text{ and } (g,1)\]
and acts on the preimage $E_0''$ of $(Q,P_i)$ in $Y_0''$ with weights
\[(-g,0),\  (0,g+1), \ (g,g+2),\text{  and }(-(g+2),1).\]
Hence $\psi_g^{-1}\times \psi_g$ acts on the image of $(Q,P_i)$ in $Y_1$ with weights
\[(-g,0),\  (0,g+1),\  \left(g,\frac{g+2}{3}\right),\text{  and }\left(-\frac{(g+2)}{3},1\right).\]

Let $U_1$ denote the coordinate patch on which $\psi_g^{-1}\times \psi_g$ acts with weights $(-g,0)$. Note that if $(z_{0,1},z_{0,2})$ were the original coordinates on $(Q,P_i)$ and blowing up twice resulted in introducing new coordinates $[z_{0,1}':z_{0,2}']$ followed by new coordinates $[z_{0,1}'':z_{0,2}'']$, then $U_1$ corresponds to the coordinate patch on $Y_1$ given by the image of the $z_{0,1}''\ne 0$ patch on $Y_0''$. On $U_1$, the fixed locus under the $(\psi_g^{-1}\times \psi_g)$--action has codimension $1$. Hence all subsequent blow-up maps $Y_i''\rightarrow Y_i'\rightarrow Y_i$ are isomorphisms on $U_1$. Let $U_c$ denote the image of $U_1$ in $Y_c$. Thus $U_c$ is obtained from $U_1$ by a sequence of $c-1$ quotients by $\mathbb{Z}/3\mathbb{Z}$ and $\psi_g^{-1}\times \psi_g$ acts on $U_c$ with weights $(-g,0)$. 

Let $V_1$ denote the  coordinate patches on which the automorphism $\psi_g^{-1}\times \psi_g$ acts with weights $\left(-\frac{(g+2)}{3},1\right)=\left(-\frac{3^{c-1}+1}{2},1\right)$. Observe that $V_1$ corresponds to the coordinate patch on $Y_1$ given by the image of the $z_{0,2''}\ne 0$ patch in $Y_0''$. Suppose each sequence of blowups $Y_i''\rightarrow Y_i'\rightarrow Y_i$ introduces new local coordinates $[z_{i,1}':z_{0,2}']$ and $[z_{i,1}'':z_{i,2}'']$. Inductively, define $V_i$ to be the image in $Y_i$ of the $z_{i-1,2}''\ne 0$ patch in $Y_{i-1}''$. Suppose $\psi_g^{-1}\times \psi_g$ acts with weights $\left(-\frac{3^{c-i}+1}{2},1\right)$ on $V_i$. Then after one blow-up $\psi_g^{-1}\times \psi_g$ acts with weights $\left(-\frac{3^{c-i}+3}{2},1\right)$ on the $z_{i,2}'\ne0$ patch. Since $-\frac{3^{c-i}+3}{2}$ is divisible by $3$, the blow-up $Y_i''\rightarrow Y_i'$ is an isomorphism on the $z_{i,2}'\ne0$ patch, and thus $\psi_g^{-1}\times \psi_g$ indeed acts with weights $\left(-\frac{3^{c-{i+1}}+1}{2},1\right)$ on $V_{i+1}$.  Hence by induction, we have that $\psi_g^{-1}\times \psi_g$ acts with weights $\left(-\frac{3^{c-i}+1}{2},1\right)$ on $V_i$ for each $1\le i \le c$ and to pass from $V_i$ to $V_{i+1}$ we perform one non-trivial blow-up followed by one $(\mathbb{Z}/3\mathbb{Z})$-quotient.


\subsection{Forms Under Quotients}\label{quotient}
Recalling the notation from the construction of $Y_c$ in Section \ref{constructionY}, consider the $3:1$ cover maps $f_i: Y_i''\rightarrow Y_{i+1}$. The Riemann-Hurwitz formula gives:
\begin{equation}\label{rhurwitz}K_{Y_i''}=f_i^*\left(K_{Y_{i+1}}+\sum_{D\in \mathrm{Div}(Y_{i+1})}\frac{a-1}{a}D\right)\end{equation}
where $a$ is the order of the group fixing the divisors in $Y_i''$ mapping to $D$ under $f_i$. 

By construction, the group $G_{i+1}/G_i$ acting on $Y_i''$ is isomorphic to $\mathbb{Z}/3\mathbb{Z}$. Thus for every irreducible divisor $D\in \mathrm{Div}(Y_{i+1})$, either $a=1$ or $a=3$. Moreover, the irreducible $D$ for which $a=3$ are exactly the images of the irreducible components of the exceptional divisors $E_i''$ obtained from the blow-up map $Y_i''\rightarrow Y_i'$, where it may happen that $E_i''\cong E_i'$. Let $E_{i,1}''\ldots, E_{i,k_i}''$ be the irreducible components of  $E_i''$. Observe that since $\eta_{i+1}$ fixes each of the $E_{i,j}''$, each component $E_{i,j}''$ descends to an irreducible divisor on $Y_{i+1}$. Equation (\ref{rhurwitz}) then yields:
\begin{equation}\label{rhurwitz2}K_{Y_i''}^{\otimes m}-\sum_{j=1}^{k_i}2mE_{i,j}'' =f_i^*\left(K_{Y_{i+1}}^{\otimes m}\right).\end{equation}

For an algebraic surface $V$ with a coordinate patch $(z_1,z_2)$ having the standard action of  $\mathbb{G}_m^2$ on $\mathbb{C}^2$, we say that a form $\omega$ is \emph{toric} on the patch $(z_1,z_2)$ if the divisor of zeros of $\omega$ on $(z_1,z_2)$ is invariant under the action of $\mathbb{G}_m^2$. 

\begin{defn} A toric form $\omega$ on a coordinate patch $(z_1,z_2)$ of an algebraic surface has \emph{vanishing sequence} $(\beta_1,\beta_2)$ on the point $(z_1,z_2)=(0,0)$ if $\omega$ vanishes to order $\beta_i$ along the hypersurface $z_i=0$.
\end{defn}

Now consider a $G_{i+1}$-invariant form $\sigma$ on $Y_i''$ which is a global section of $K_{Y_i''}^{\otimes m}$.  Suppose $Y_i''$ has local coordinates $(z_1,z_2)$ around some $E_{i,j}''$ fixed by the action of $G_{i+1}$ such that, without loss of generality $E_{i,j}''$, is given by $z_1=0$. Consider the point  $R=(0,0)$ on $E_{i,j}''$ and suppose the vanishing sequence of $\sigma$ on $R$ is
$(\alpha_1,\alpha_2).$ Then, using Equation (\ref{rhurwitz2}), the vanishing sequence of the pushforward of $\sigma$ to $Y_{i+1}$ has vanishing sequence on the image of $R$ in $Y_{i+1}$ given by
\begin{equation}\label{quotienteq}\left(\frac{1}{3}(\alpha_1-2m), \alpha_2\right).\end{equation}

\subsection{Forms Under Blow-Ups}\label{blowup}

Let $V$ be an algebraic surface and let $\sigma$ be a global section of $K_{V}^{\otimes m}$. For local coordinates $(z_1,z_2)$, consider the blow-up of $V$ at the point $(0,0)$. Suppose $\sigma$ has vanishing sequence $(\alpha_1,\alpha_2)$ on the point $(0,0)$. Then $\sigma$ is given locally by:
$$f(z_1, z_2)(dz_1 dz_2)^{\otimes m},$$
where $f$ has vanishing sequence $(\alpha_1,\alpha_2)$ on $(0,0)$. 

Then, blowing up $V$ at $(0,0)$ introduces new coordinates $[z_1': z_2']$, with $z_1z_2'=z_2z_1'$. Hence on the coordinate patch of the blown-up variety $V'$ given by $z_1'\ne0$, we have coordinates $(z_1, z_2')$, where $z_2=z_1z_2'$.  So locally around the exceptional divisor $E$, the form $\sigma$ pulls back to the form:
\[z_1^{m}f(z_1,z_1z_2)(dz_1dz_2')^{\otimes m}.\]

In the new coordinates $(z_1,z_2'),$ consider the point $R=(0,0)$. Then the vanishing sequence on $R$ of the pullback of $\sigma$ to $V'$ is given by
\begin{equation}\label{blowupeq}(\alpha_1+\alpha_2 + m,\alpha_2).\end{equation}

\subsection{Vanishing of Forms on Type II Singular Points}

Consider a form $\sigma$ with vanishing sequence $(\alpha_1,\alpha_2)$ on a Type II fixed point in $C_g\times C_g$ of the form $(Q,P_i)$. Then, using the notation of Section \ref{typeIIsingpoints section}, Equations \eqref{blowupeq} and \eqref{quotienteq} yield that $\sigma$ has vanishing sequence $\left(\frac{1}{3}(\alpha_1+2\alpha_2),\alpha_2\right)$ at the origin of $U_1$ and vanishing sequence $\left(\alpha_1, \frac{1}{3}(2\alpha_1+\alpha_2)\right)$ at the origin of $V_1$. 

Since to obtain $U_c$ from $U_1$ one performs a sequence of $c-1$ quotients by $\mathbb{Z}/3\mathbb{Z}$, the vanishing sequence of $\sigma$ at the origin of $U_c$ is given by
\begin{equation}\label{vanishing1}\left(\frac{1}{3}\big(\cdots\big(\frac{1}{3}\big(\frac{1}{3}\big(\frac{1}{3}(\alpha_1+2\alpha_2)-2m)-2m\big)\cdots\big)-2m\big),\alpha_2\right)=\left(\frac{1}{3^c}(\alpha_1+2\alpha_2-m(3^c-3)),\alpha_2\right).\end{equation}

Since to obtain $V_c$ from $V_1$ one performs $c-1$ iterations of a single non-trivial blow-up followed by a $(\mathbb{Z}/3\mathbb{Z})$-quotient, the vanishing sequence of $\sigma$ on the origin of $V_c$ is given by
\begin{equation}\label{vanishing2}\left
(\alpha_1, 
\frac{1}{3}\big( \cdots \big( \frac{1}{3}\big(\frac{1}{3}(2\alpha_1+\alpha_2)+\alpha_1-m\big)\cdots+\alpha_1-m\big)\right)=\left(\alpha_1,\frac{1}{3^c}\big(\frac{3^c+1}{2}\alpha_1+\alpha_2-m\frac{3^c-3}{2}\big)\right).\end{equation}

Symmetrically, if $\sigma$ has vanishing sequence $(\beta_1,\beta_2)$ on a Type II fixed point in $C_g\times C_g$ of the form $(P_i,Q)$, then the chain of rational curves in $Y_c$ resolving this singular point of $(C_g\times C_g)/\langle \psi_g^{-1}\times \psi_g \rangle$ has one endpoint (the one covered by coordinate patch $z_{c,1}\ne0$) on which $\sigma$ has vanishing sequence 
\begin{equation}\label{vanishing3}\left(\frac{1}{3^c}\big(\beta_1+ \frac{3^c+1}{2}\beta_2+-m\frac{3^c-3}{2}\big), \beta_2\right)\end{equation}
and one endpoint (the one covered by coordinate patch $z_{c,2}\ne0$) on which $\sigma$ has vanishing sequence
\begin{equation}\label{vanishing4}\left(\beta_1, \frac{1}{3^c}(2\beta_1+\beta_2-m(3^c-3))\right).\end{equation}

\subsection{Kodaira Dimension Computation for $X_c$}
We now compute the plurigenera $P_m=h^0(Y_c,K_{Y_c}^{\otimes m})$ of the variety $Y_c$  using the following theorem of K\"{o}ck and Tait.
\begin{thm}\label{global}\cite[Theorem 5.1]{kock} Let $C$ be a hyperelliptic curve of genus $g\ge 2$ of the form $y^2=f(x)$ for some $f$ and let $\omega\in K_C^{\otimes m}$ be given by $\omega=\frac{dx^{\otimes m}}{y^m}$. Then an explicit basis for $H^0(C,K_C^{\otimes m})$ is given by the following:
$$\begin{cases}
\omega, x\omega, \ldots, x^{g-1}\omega & \mbox{if } m=1\\
\omega, x\omega, x^2\omega & \mbox{if }m=2 \mbox{ and } g=2\\
\omega, x\omega, \ldots, x^{m(g-1)}\omega; y\omega, xy\omega, \ldots, x^{(m-1)(g-1)-2}y\omega & \mbox{otherwise}
\end{cases}.$$
\end{thm}

\begin{prop}
\label{dim2} For $c\ge2$, the surface $X_c$ has Kodaira dimension $1$.\end{prop}

\begin{proof}
Fix some $m>0$. By Theorem \ref{global}, we are interested in global sections of $K_{C_g}^{\otimes m}$ of the form $x^a\omega$, where $0\le a\le m(g-1)$ or of the form $x^ay\omega$, where $0\le a \le (m-1)(g-1)-2$. 
Begin by considering the affine patch of $C_g$ given by $\{y^2=x^{2g+1}+1\}$.  We know the variable $x$ is a local coordinate for $C_g$ at the points $P_1$ and $P_2$. Since we have
$$\omega=\frac{dx^{\otimes m}}{y^{m}},$$
the form $\omega$ has order of vanishing equal to $0$ at $P_1$ and $P_2$. Hence both forms $x^{a}\omega$ and $x^{a}y\omega$ have order of vanishing $a$ at fixed points $P_1$ and $P_2$. 

In $(u,v)$-coordinates, the form $\omega$ is given by
$$\frac{(-1)^{m}u^{m(g-1)}du^{\otimes m}}{v^{m}}$$
Recall that $v$ is a local coordinate near the point $Q$.  The equation $v^2=u^{2g+2}+u$ yields $2v\cdot dv=((2g+2)u^{2g+1}+1)\cdot du$. Therefore $du$ and $v$ vanish to the same order. Moreover, $u$ has order of vanishing $2$ with respect to $v$, hence the order of vanishing of $\omega$ at the point $Q$ is $2m(g-1)=m(3^{c}-3)$. 

Hence a form $x^{a}\omega=u^{-a}\omega$ has order of vanishing at $Q$ given by 
$$m(3^{c}-3)-2a$$
and a form $x^{a}y\omega=u^{-(a+g+1)}v\omega$ has order of vanishing at $Q$ given by
$$2m(g-1)-2(a+g+1)+1=m(3^{c}-3)-2a-3^c.$$

Now, without loss of generality a  global section $\sigma$ of $K_{C_g\times C_g}^{\otimes m}$ is of three possible forms:
\begin{enumerate}
\item $x_1^{a_1}\omega_1\times x_2^{a_2}\omega_2$
\item $x_1^{a_1}\omega_1\times x_2^{a_2}y_2\omega_2$
\item $x_1^{a_1}y_1\omega_1\times x_2^{a_2}y_2\omega_2.$
\end{enumerate}

Suppose that the form $\sigma$ corresponds to some global section of $K_{Y_c}^{\otimes m}$.  Write $(\alpha_1,\alpha_2)$ for the vanishing sequence of $\sigma$ on some Type II fixed point of the form $(Q,P_i)$ on $C_g\times C_g$ and $(\beta_1,\beta_2)$ for the  vanishing sequence of $\sigma$ on some Type II fixed point of the form $(P_i,Q)$ on $C_g\times C_g$. Then by Equation \eqref{vanishing1} we must have
\begin{equation}\label{alphaineq}\alpha_1+2\alpha_2-m(3^c-3)\ge 0 \ \text{ and } \ \alpha_2\ge 0.\end{equation}
Similarly, by Equation \eqref{vanishing4} we must have
\begin{equation}\label{betaineq}\beta_1\ge 0 \ \text{ and } \ 2\beta_1+\beta_2-m(3^c-3)\ge 0.\end{equation}

First consider when $\sigma$ is of the form $x_1^{a_1}\omega_1\times x_2^{a_2}\omega_2$. Then we know
\[(\alpha_1,\alpha_2)=(m(3^{c}-3)-2a_1,a_2) \ \text{ and }\ 
(\beta_1,\beta_2)=(a_1, 3(3^{c}-3)-2a_2).\]
Hence after simplification, Equations \eqref{alphaineq} and \eqref{betaineq} yield that we must have $a_1=a_2$. 

In the case of a global section $\sigma$ of the form $x_1^{a_1}\omega_1\times x_2^{a_2}y_2\omega_2$, we have:
\[
(\alpha_1,\alpha_2)=(3m(3^{c-1}-1)-2a_1,a_2)
\ \text{ and }\ 
(\beta_1,\beta_2)= (a_1, 3m(3^{c-1}-1)-2a_2-3^c).\]
Hence after simplification, Equations \eqref{alphaineq} and \eqref{betaineq} yield that we must have
$2a_2\ge 2a_2+3^c,$
which is impossible. So no such $\sigma$ can exist.

Finally, in the case of a global section $\sigma$ of the form $x_1^{a_1}y_1\omega_1\times x_2^{a_2}y_2\omega_2$, we have:
\[
(\alpha_1,\alpha_2)=(3m(3^{c-1}-1)-2a_1-3^c,a_2)
\ \text{ and } \ 
(\beta_1,\beta_2)= (a_1, 3m(3^{c-1}-1)-2a_2-3^c).
\]
Hence after simplification, Equations \eqref{alphaineq} and \eqref{betaineq} yield that we must have
$2a_2\ge 2a_2+2\cdot3^c,$
which is impossible. So no such $\sigma$ can exist.

Therefore we have shown that the only global sections of $K_{C_g\times C_g}^{\otimes m}$ that can correspond to global sections of $K_{Y_c}^{\otimes m}$ are those of the form $x_1^a\omega_1\times x_2^a\omega_2$
for $0\le a \le m(g-1)$.  Moreover, note that sections of this form are always $\psi_g^{-1}\times \psi_g$-invariant, therefore any global section of $K_{Y_c}^{\otimes m}$ and hence any global section of $K_{X_c}^{\otimes m}$ corresponds to a form $x_1^a\omega_1\times x_2^a\omega_2$
for $0\le a \le m(g-1)$. This is a linear condition on $m$, therefore the Kodaira dimension of $X_c$ is at most equal to $1$.

However, by construction, we have $h^0(X_c,K_{X_c})=h^{2,0}=g$.  In particular, this means $h^0(X_c,K_{X_c})$ is greater than $1$, so the Kodaira dimension of $X_c$ is at least equal to $1$. Hence, the Kodaira dimension of $X_c$ is exactly equal to $1$.

\end{proof}


\section{The Canonical Bundle $K_{X_c}$}\label{basepointfreesection}
Observe that the proof of Theorem \ref{dim2} yields that $H^0(X_c,K_{X_c})$ has basis corresponding to the set of global sections of $K_{C_g\times C_g}$ given by $B=\{x_1^a\omega_1\times x_2^a\omega_2\mid 0\le a\le g-1\}$, where $x_i$ denotes the $x$-coordinate of the $i$-th factor in the product $C_g\times C_g$ and $\omega_i\coloneqq\frac{dx_i}{y_i}$. 

\begin{prop}\label{basepointfree}For $c\ge 2$, the canonical bundle $K_{X_c}$ is basepoint free.
\end{prop}

\begin{proof}
Consider the rational map $f\colon X_c\dashrightarrow \mathbb{P}(H^0(X_c,K_{X_c}))$ induced by the canonical bundle $K_{X_c}$, where $\mathbb{P}(H^0(X_c,K_{X_c}))$ denotes the space of hyperplanes in $H^0(X_c,K_{X_c})$. Observe that $f$ fits into a diagram
\begin{equation}
\begin{tikzcd}
X_c\rar[-,dashrightarrow]{f}&\mathbb{P}(H^0(X_c,K_{X_c}))\\
C_g\times C_g\uar[-,dashrightarrow]\rar[-,dashrightarrow] &\mathbb{P}(H^0(C_g^2,K_{C_g^2}))\uar[-,dashrightarrow],
\end{tikzcd}
\end{equation}
where the horizontal maps are given by evaluation on the basis $B$ and the rational vertical map on the left is the sequence of blow-ups, blow-downs, and quotients needed to obtain $X_c$ from $C_g\times C_g$. 

Let $s_a\coloneqq x_1^a\omega_1\times x_2^a\omega_2$ for $0\le a\le g-1$. Observe that the points of $C_g\times C_g$ on which all the $s_a$ vanish are exactly the Type II fixed points. Thus to prove that $K_{X_c}$ is basepoint free, we just need to ensure that not all of the $s_a$ vanish on the image in $X_c$ of a Type II fixed point.

Without loss of generality consider a Type II fixed point $A$ of the form $(Q,P_i)$ for some $i\in \{1,2\}$, noting that a symmetric argument will work for Type II fixed points of the form $(P_i,Q)$.  As discussed in Section \ref{typeIIsingpoints section}, the image of $A$ in $Y_c$ is a chain of rational curves where one endpoint of this chain is covered by coordinate patch $U_c$ and the other endpoint of the chain is covered by coordinate patch $V_c$. Moreover, as we computed in the proof of Proposition \ref{dim2},  the form $s_a$ has vanishing sequence $(m(3^{c}-3)-2a,a)$ on $A$. Hence, by Equations \eqref{vanishing1} and \eqref{vanishing2}, the form $s_a$ has vanishing sequences 
\[(0,a)\  \text{ and } \ (3^c-3-2a, \frac{1}{2}(3^c-3-2a))\]
at the origins of $U_c$ and $V_c$ respectively.

By Section \ref{blowup}, if $s_a$ has vanishing sequence $(\alpha_1,\alpha_2)$ at some point $Z$, then the exceptional curve $E$ resulting from blowing up $Z$ will be covered by two coordinate patches such that $s_a$ has vanishing sequence $(\alpha_1+\alpha_2+1, \alpha_2)$ at the origin of one and vanishing sequence $(\alpha_1, \alpha_1+\alpha_2+1)$ at the origin of the other. 
Thus, consider the chain of rational curves in $Y_c$ resolving the point $A$. When passing to $X_c$, we know all but the two outer curves in the chain get contracted. Thus the resulting two rational curves in $X_c$ are covered by three coordinate patches: $U_c$ on one end, $V_c$ on the other end, and a third patch we denote by $W_c$ in the middle. Suppose $s_a$ has vanishing sequence $(\gamma_1,\gamma_2)$ on the origin of $W_c$. Since $s_a$ has vanishing sequence $(0,a)$ on the origin of $U_c$, we must have $\gamma_2=0$. Similarly, since $s_a$ has vanishing sequence $(3^c-3-2a, \frac{1}{2}(3^c-3-2a))$ at the origin of $V_c$, we must have $\gamma_1=\frac{1}{2}(3^c-3-2a)$. It follows that $s_a$ has vanishing sequence at the origins of the three coordinate patches $U_c$, $W_c$, and $V_c$ given respectively by
\[(0,a), \ \left(\frac{1}{2}(3^c-3-2a), 0\right), \  \text{ and } \ \left(3^c-3-2a, \frac{1}{2}(3^c-3-2a)\right).\]

Hence on $U_c$, we know $s_0\ne 0$ and on $W_c$ and $V_c$ we know $s_{g-1}\ne0$. Thus, indeed, not all the $s_a$ vanish on the image in $X_c$ of $A$ and so $K_{X_c}$ is indeed basepoint free. 
\end{proof}


\section{The Elliptic Fibration $f\colon X_c\rightarrow \mathbb{P}^1$}\label{ellipticfibrationsection}
We have shown in Proposition \ref{basepointfree} that the canonical bundle $K_{X_c}$ is basepoint free and in Proposition \ref{dim2} that the Kodaira dimension of $X_c$ is $1$. Hence the Iitaka fibration
\[f\colon X_c\rightarrow \mathbb{P}(H^0(X_c,K_{X_c}))\]
given by sending a point $x$ to its evaluation on the basis $B$ of global sections of $K_{X_c}$ has elliptic curve fibers and has image a curve \cite[Theorem 2.1.33]{positivity}). Moreover, since $X_c$ has Hodge numbers $h^{1,0}=h^{0,1}=0$, this image curve must have genus $0$. Therefore, the surface $X_c$ is an elliptic surface equipped with elliptic fibration $f:X_c\rightarrow \mathbb{P}^1$. One may find many of the basic properties of elliptic surfaces in the surveys \cite{miranda} and \cite{elsur}.

We now study in detail the geometry of the elliptic fibration $f\colon X_c\rightarrow \mathbb{P}^1$.  In particular, we extensively use Kodaira's classification, in \cite{kodaira} and \cite{kodaira2}, of the possible singular fibers of an elliptic surface. For a survey of the possible fiber types, see \cite[I.4]{miranda} and \cite[Section 4]{elsur}.

As we will see, the two kinds of singular fibers that appear in the fibration $f$ are singular fibers of type $I_b$ for $b>0$ and singular fibers of type $I_b^*$ for $b\ge 0$. Singular fibers of type $I_b$ consist of $b$ smooth rational curves meeting in a cycle, namely meeting with dual graph the affine Dynkin diagram $\tilde{A}_b$. Singular fibers of type $I_b^*$ consist of $b+5$ smooth rational curves meeting with dual graph the affine Dynkin diagram $\tilde{D}_{b+4}$.

Recall from Section \ref{quotsingsection} that the quotient $(C_g\times C_g)/\langle\psi_g^{-1}\times \psi_g\rangle$ has two types of singular points: Type I and Type II. The five Type I singular points are all DuVal singularities of type $A_{3^c-1}$ and thus each has resolution in $X_c$ consisting of a chain of $3^c-1$ rational curves of self-intersection $-2$. The four Type II singular points each have a resolution in $X_c$ consisting of a chain of $2$ rational curves: one with self-intersection $-2$ and one with self-intersection $-(g+1)$.

Let $\delta_1=Q\times P_1$, $\delta_2=Q\times P_2$, $\delta_3=P_1\times Q$, and $\delta_4=P_2\times Q$ denote these four Type II singular points and let $\mathcal{S}_1,$ $\mathcal{S}_2$, $\mathcal{S}_3$, and $\mathcal{S}_4$ denote each of their respective $-(g+1)$-curves in $X_c$. 

\begin{thm}\label{sing}For $c\ge 2$, the elliptic surface $f\colon X_c\rightarrow \mathbb{P}^1$ has $3^c+2$ singular fibers: one of type $I_{4\cdot3^c}$ located at $0$, one of type $I_{3^c}^*$ located at $\infty$, and the remaining $3^c$ of type $I_1$ and located at the points $\zeta^i$, for $\zeta$ a primitive $3^c$-th root of unity.  Additionally, each of the rational curves $\mathcal{S}_1,$ $\mathcal{S}_2$, $\mathcal{S}_3$, and $\mathcal{S}_4$ coming from the resolution of a Type II singular point corresponds to a section of $f$. 
\end{thm}

\begin{figure}[h]
\centering{
\includegraphics
[width=.9\textwidth]
{drawing3.pdf}
\caption{The elliptic surface $f\colon X_c\rightarrow \mathbb{P}^1$}}
\end{figure}

\begin{proof}
As in the proof of Proposition \ref{basepointfree}, we have the diagram
\begin{equation}\label{maindiagram}
\begin{tikzcd}
X_c\arrow{r}{f}&\mathbb{P}(H^0(X_c,K_{X_c}))\\
C_g\times C_g\uar[-,dashrightarrow]\rar[-,dashrightarrow] &\mathbb{P}(H^0(C_g^2,K_{C_g^2}))\uar[-,dashrightarrow], 
\end{tikzcd}
\end{equation}
where the horizontal maps are given by evaluation on the basis of global sections $B=\{x_1^a\omega_1\times x_2^a\omega_2\mid 0\le a\le g-1\}$ and the rational vertical map on the left is the sequence of blow-ups, blow-downs, and quotients needed to obtain $X_c$ from $C_g\times C_g$. 

To understand the fibration $f:X_c\rightarrow \mathbb{P}^1\subset \mathbb{P}(H^0(X_c,K_{X_c}))$, we thus first would like to understand the composition
\[\alpha\colon C_g\times C_g\dashrightarrow \mathbb{P}(H^0(X_c,K_{X_c})).\]

Let $s_a\coloneqq (x_1x_2)^a \omega_1\omega_2$ for $0\le a\le g-1$ denote the global sections of $K_{C_g\times C_g}$ that correspond to a basis of global sections of $K_{X_c}$. We know that  the points of $C_g\times C_g$ on which all the $s_a$ vanish are exactly the Type II fixed points. Since the map $\alpha\colon C_g\times C_g\dashrightarrow \mathbb{P}(H^0(X_c,K_{X_c}))$ may be viewed as the rational map given by
\[(z_1,z_2)\mapsto [s_0(z_1,z_2): \cdots : s_{g-1}(z_1,z_2)],\]
we have
\begin{equation}\label{preimage}
\begin{aligned}
\alpha^{-1}([1:0:\cdots:0])=&(P_1\times (C_g-Q))\cup (P_2\times (C_g-Q))\cup ((C_g-Q)\times P_1)\cup ((C_g-Q)\times P_2)\\
\alpha^{-1}([0:\cdots:0:1])=&(Q\times Q)\cup (Q\times (C_g-P_1-P_2))\cup ((C_g-P_1-P_2)\times Q).
\end{aligned}
\end{equation}

In particular, the image of the fixed points in $C_g\times C_g$ of the form $(P_i,P_j)$ is the point $[1:0:\cdots:0] $ in  $\mathbb{P}(H^0(X_c,K_{X_c}))$. Each such point has image in $X_c$ consisting of a chain of $3^c-1$ rational curves and so by the diagram (\ref{maindiagram}), the fibration $f\colon X_c\rightarrow \mathbb{P}^1$ must send all of these $3^c-1$ rational curves to the point $[1:0:\cdots:0]$.

Moreover, using \eqref{maindiagram} in conjunction with \eqref{preimage}, since $f$ is a morphism we must have that the strict transforms in $X_c$ of the curves $P_1\times C_g,$ $P_2\times C_g,$  $C_g\times P_1,$  and $C_g\times P_2$
also get sent to $[1:0:\cdots:0]$. Note that the strict transform of $C_g\times P_j$ will intersect the chain of rational curves resolving the singularity $P_i\times P_j$ at one end of the chain and the strict transform of $P_i\times C_g$ will intersect the chain at the other end of the chain. 

Similarly, by \eqref{preimage} the image of the fixed point $(Q,Q)$ in $C_g\times C_g$ will be sent by $\alpha$ to the point $[0:\cdots:0:1] $ in  $\mathbb{P}(H^0(X_c,K_{X_c}))$. Since such a fixed point has image in $X_c$ consisting of a chain of $3^c-1$ rational curves, diagram (\ref{maindiagram}) yields that $f$ must send all of these rational curves to the point $[0:\cdots:0:1]$.

Moreover, by \eqref{maindiagram} in conjunction with \eqref{preimage}, since $f$ is a morphism we must have that the strict transforms in $X_c$ of the curves 
$Q\times C_g $ and $C_g\times Q$
 get sent to $[0:\cdots:0:1]$ as well. Again, the strict transform of $Q\times C_g$ will intersect the chain of rational curves resolving the singularity  $(Q,Q)$ at one end and the strict transform of $C_g\times Q$ will intersect the chain at the other end. 
 
Therefore we have established that $f$ sends the strict transforms in $X_c$ of the curves
 \begin{equation}\label{Pcurves}P_1\times C_g,\  P_2\times C_g, \ C_g\times P_1, \text{ and } C_g\times P_2\end{equation}
 to the point $[1:0:\cdots:0]$ in $\mathbb{P}(H^0(X_c,K_{X_c}))$ and the strict transforms of the curves 
\begin{equation}\label{Qcurves}Q\times C_g \text{ and } C_g\times Q\end{equation}
 to the point $[0:\cdots:0:1]$.
 
Each of the four Type II fixed points $\delta_1,$ $\delta_2$, $\delta_3,$ and $\delta_4$ in $C_g\times C_g$ has one of the curves in (\ref{Pcurves}) and one of the curves in (\ref{Qcurves}) passing through it. Recall that each point $\delta_j$ has image in $X_c$ consisting of a chain of $2$ rational curves: a $(-2)$-curve and the $-(g+1)$-curve $\mathcal{S}_j$. Observe that for a Type II point of the form $Q\times P_i$, the strict transform in $X_c$ of the curve $Q\times C_g$ will intersect this chain of rational curves at the end of the $(-2)$-curve and the strict transform of the curve $C_g\times P_i$ will intersect the chain of rational curves at the end of the $-(g+1)$-curve.

By the adjunction formula, the irreducible curves on $X_c$ contained in a fiber of $f:X_c\rightarrow \mathbb{P}^1\subset \mathbb{P}(H^0(X_c,K_{X_c}))$ are exactly those curves $F\subset X_c$ with $K_{X_c}.F=0$. So, of the two curves in the resolution of a Type II singular point $\delta_j$, the curve with self-intersection $-2$ gets mapped to a point by $f$ and the curve $\mathcal{S}_i$ with self-intersection $-(g+1)$ gets mapped to all of $\mathbb{P}^1$ by $f$.

Since the $(-2)$-curve in the resolution of each Type II singular point intersects either the curve $Q\times C_g$ or the curve $ C_g\times Q$, both of which get sent to the point $[0:\cdots:0:1]$ by $\alpha$, we have that $f$ sends these $(-2)$-curves to $[0:\cdots:0:1]$ as well. 

Thus we have accounted for all of the curves in the resolutions of the singular points of $(C_g\times C_g)/\langle\psi_g^{-1}\times \psi_g\rangle$ and described their images under $f$. To summarize, the fiber in $X_c$ of the point $[1:0:\cdots:0]$ in $ \mathbb{P}(H^0(X_c,K_{X_c}^{\otimes m}))$ is a cycle consisting of the four sets of $3^c-1$ rational curves coming from the resolutions of the points $P_i\times P_j$ together with the four curves in (\ref{Pcurves}). Hence, the fiber under $f$ of $[1:0:\cdots:0]$ consists of 
$4(3^c-1)+4=4\cdot 3^c$
rational curves.  This is a fiber of type $I_{4\cdot 3^c}$ in Kodaira's classification of the singular fibers of an elliptic surface. Similarly, the fiber in $X_c$ of the point $[0:\cdots:0:1]$ consists of the $3^c-1$ rational curves resolving the singularity $Q\times Q$ together with the curves in (\ref{Qcurves}) and the four $(-2)$-curves, each from a resolution of a Type II point. Hence, the fiber under $f$ of $[0:\cdots:0:1]$  consists of a chain of $3^c+1$ rational curves, where each curve on the ends of the chain has two additional curves coming off it.  This is a fiber of type $I_{3^c}^*$.

We now want to identify the singular fibers of $f$ away from the points $0$ and $\infty$ in $\mathbb{P}^1$. Therefore we would like to understand the map 
$f\colon X_c\rightarrow   \mathbb{P}^1\subset \mathbb{P}(H^0(X_c,K_{X_c}))$
 away from any points on which a global section of $K_{X_c}$ vanishes. Say that $\mathbb{P}(H^0(X_c,K_{X_c}))$ has coordinates $[w_0:\cdots:w_{g-1}]$. Then on the affine patch of $\mathbb{P}(H^0(X_c,K_{X_c}))$ given by $w_0\ne0$, the image of $\alpha$ is of the form 
$(t,t^2,\ldots, t^{g-1}),$ where $t=x_1(z_1) x_2(z_2)$. The fibers of $\alpha$ on this subset are then the curves $F_t\subset C_g\times C_g$ defined by:
\[(y_1^2=x_1^{2g+1}+1,\  y_2^2=x_2^{2g+1}+1,\  x_1x_2=t),\]
where we are assuming $t\ne0$. Such curves $F_t$ have Jacobian
\begin{equation*}
\left(\begin{array}{cccc}
(2g+1)x_1^{2g} &2y_1 & 0&0\\
0&0&(2g+1)x_2^{2g} &2y_2\\
x_2&0&x_1&0
\end{array}\right)
\end{equation*}

Hence the curve $F_t$ is singular when both $y_1$ and $y_2$ are equal to $0$. When this is the case, then we have $x_1=\zeta^i$, $x_2=\zeta^j$, and $t=\zeta^{i+j}$, for some $i$ and $j$. In other words, if $t\in \mathbb{C}^*$ is any $3^c$-th root of unity, then the fiber $F_t$ is singular and has singularities at the points
$(x_1,y_1,x_2,y_2)=(\zeta^i,0,\zeta^j,0)$
such that $i+j\equiv 0 \mod3^c$. It follows that the curve $F_t$ has $3^c$ singular points. Since the action of $\psi_g^{-1}\times \psi_g$ permutes these $3^c$ singular points, the image $\bar{F}_t$ of $F_t$ in $X_c$ is a curve with a single singularity. So $\bar{F}_t$ is then an irreducible singular fiber of $f$. By Kodaira's classification of the possible singular fibers of an elliptic surface, the fiber $\bar{F}_t$ must in fact be a rational nodal curve, namely a fiber of type $I_1$.

Therefore, in summary, we have identified a singular fiber of  type $I_{4\cdot3^c}$ at the point $0$ in $\mathbb{P}^1$, we have identified a singular fiber of type $I_{3^c}^*$ at the point $\infty$, and we have identified $3^c$ singular fibers of type $I_1$, located at the points $\zeta^i$. We may confirm that these are indeed all of the singular fibers of $f$ using the following description of the topological Euler number of an elliptic surface. From \cite[Proposition 5.16]{cossec},
for a complex elliptic surface $\varphi:S\rightarrow C$ with fiber $F_v$ at $v\in C$ having $m_v$ components, we have
\begin{equation}\label{euler}\chi_{\mathrm{top}}(S)=\sum_{v\in C}e(F_v),\end{equation}
where
\begin{equation*}
e(F_v)=
\begin{cases}
0 & \mbox{if } F_v \text{ is smooth}\\
m_v & \mbox{if } F_v \text{ is of type }I_n\\
m_v+1&\mbox{otherwise}.
\end{cases}
\end{equation*}

For any elliptic surface, Noether's formula implies that the topological Euler number is $12$ times the geometric Euler number. Since the surface $X_c$ has irregularity $q=0$ and geometric genus $p_g=g$, we thus have 
$\chi_{\mathrm{top}}(X_c)=12(g+1)=6\cdot3^c+6.$
Considering the fibration $f:X_c\rightarrow \mathbb{P}^1$, the $4\cdot3^c$ components from the singular fiber of type $I_{4\cdot3^c}$, the $3^c+5$ components from the singular fiber of type $I_{3^c}^*$, and the $3^c$ components from the $3^c$ singular fibers of type $I_1$ account for exactly $6\cdot3^c+6$ on the right hand side of Equation (\ref{euler}). So we have indeed found all the singular fibers  of $f$.

It remains to verify that the curves $\mathcal{S}_1$, $\mathcal{S}_2$, $\mathcal{S}_3$, and $\mathcal{S}_4$ in the resolutions of the Type II singular points $\delta_1,$ $\delta_2,$ $\delta_3,$ and $\delta_4$ indeed correspond to sections of $f$. We know that $f$ maps each curve $\mathcal{S}_i$ surjectively onto $\mathbb{P}^1$, therefore we just need to check that for each $t\in \mathbb{P}^1$ there is a unique point $s\in \mathcal{S}_i$ such that $f(s)=t$. 

Without loss of generality suppose that $\delta_j$ is of the form $Q\times P_i$,  as a symmetric argument will work for points of the form $P_i\times Q$. Let $t\in \mathbb{P}^1$ and consider the fiber $\overline{F}_t=f^{-1}(t)$. Observe that $\mathcal{S}_j$ intersects $\overline{F}_0$ at the point of intersection with the strict transform of the curve $C_g\times P_i$ and intersects $\overline{F}_\infty$ at the point of intersection with the $(-2)$-curve in the resolution of $\delta_j$. Hence we may assume $t\ne0,\infty$. 

Now $\overline{F}_t$ is the image in $X_c$ of the curve $F_t$ in $C_g\times C_g$ given by the equation $x_1x_2=t$.  Recall that $Q\times P_i$ is given in local coordinates by $(v_1,x_2)=(0,0)$ and so near $\delta_j$ the curve $F_t$ is given by $u_1^{-1}x_2=t$.  We may rewrite this as $\gamma(v_1)^{-1}x_2=t$, for $\gamma(v_1)$ a continuous function of degree $2$ in $v_1$. Hence, close to $\delta_j$, the curve $F_t$ has coordinates given by $(v_1,t\gamma(v_1))$.  Thus the slope of $F_t$ at $\delta_j$ is 
$\lim_{v_1\to 0}\frac{t\gamma(v_1)}{v_1}=0.$

It follows that the strict transform $F_t'$ of $F_t$ in the blow-up of $C_g\times C_g$ at $\delta_j$ intersects the exceptional curve $E_0'$ with coordinates $[z_{0,1}':z_{0,2}']$ at the point $[z_{0,1}':z_{0,2}']=[1:0]$. Taking the coordinate patch $z_{0,1}'\ne0$ yields local coordinates $(v_1,z_{0,2}')$ and $F_t'$ intersects $E_0'$ at the point $(v_1,z_{0,2}')=(0,0)$. Moreover, near this point, the curve $F_t'$ has coordinates $(v_1,z_{0,2}')=\left(v_1, \frac{t\gamma(v_1)}{v_1}\right)$, since $v_1z_{0,2}'=z_{0,1}'x_2$. Hence the slope of $F_t'$ at the point $(v_1,z_{0,2}')=(0,0)$ is 
$\lim_{v_1\to 0}\frac{t\gamma(v_1)}{v_1^2}=t.$

Since the point $(v_1,z_{0,2}')=(0,0)$ gets blown up in the transformation $Y_0''\rightarrow Y_0$, the strict transform $F_t''$ of the curve $F_t'$ after this blowup intersects the exceptional curve at the point with coordinate $t$. Moreover, observe that this point with coordinate $t$ is covered by the coordinate patch $U_1$ introduced in Section \ref{typeIIsingpoints section}. Since the coordinate patch $U_c$ in $Y_c$ is obtained from $U_1$ by a sequence of $c-1$ quotients by $\mathbb{Z}/3\mathbb{Z}$, this point corresponds to the point with coordinate $t^{3^{c-1}}$ on the image of this exceptional curve in $Y_c$. But since this exceptional curve does not get contracted in passing from $Y_c$ to $X_c$ (see the proof of Proposition \ref{basepointfree}), this intersection point is also the point with coordinate $t^{3^{c-1}}$ on the image of this exceptional curve in $X_c$, which is just the rational curve $\mathcal{S}_j$

In other words, the curve $\overline{F}_t''$ intersects $\mathcal{S}_j$ at the point of $\mathcal{S}_j$ with coordinate $t^{3^{c-1}}$. Hence $s=t^{3^{c-1}}$ is the unique point in $\mathcal{S}_j$ such that $f(s)=t$ and therefore $\mathcal{S}_j$ indeed corresponds to a section of $f$.
\end{proof}


\section{The Surface $X_c$ is Extremal}\label{S6}
The Mordell-Weil group of an elliptic fibration $\varphi:S\rightarrow C$ is the group of $K$-rational points on the generic fiber of $\varphi$, where $K=\mathbb{C}(C)$. Such an elliptic surface $S$ is called \emph{extremal} if it has maximal Picard rank $\rho(S)$, meaning $\rho(S)=h^{1,1}(S)$, and its Mordell-Weil group has rank $r=0$. 

\begin{cor}\label{extremal} For $c\ge 2$, the surface $f:X_c\rightarrow \mathbb{P}^1$ is an extremal elliptic surface.\end{cor}

\begin{proof}
For the fibration $f:X_c\rightarrow \mathbb{P}^1$ and for any $v\in \mathbb{P}^1$, let $F_v$ denote the fiber $f^{-1}(v)$ and let $m_v$ denote the number of components of $F_v$. Define 
\[R=\{v\in \mathbb{P}^1\mid  F_v \text{ is reducible} \}.\]
The Shioda-Tate formula \cite[Corollary 6.13]{elsur} expresses the Picard number $\rho(X_c)$ in terms of the reducible singular fibers and the Mordell-Weil group of $f:X_c\rightarrow \mathbb{P}^1$:
\begin{equation}\label{stformula}\rho(X_c)=2+\sum_{v\in R}(m_v-1)+r.\end{equation}

We know from Theorem \ref{sing} that $f$ has two reducible singular fibers: one of type $I_{4\cdot3^c}$ at $0$ and one of type $I_{3^c}^*$ at $\infty$, having $4\cdot 3^c$ and $3^c+5$ components respectively. Therefore
\[\sum_{v\in R}(m_v-1)=(4\cdot 3^c-1)+(3^c+4)=5\cdot 3^c+3.\]
So then Equation (\ref{stformula}) becomes
$\rho(X_c)=5\cdot 3^c+5+r.$

However we know the Picard number $\rho(X_c)$ of $X_c$ satisfies $\rho(X_c)\le h^{1,1}(X_c)$. Moreover, we showed in the proof of Theorem \ref{sing} that $\chi_{\mathrm{top}}(X_c)=12(g+1)=6\cdot3^c+6$. Since $h^{1,0}(X_c)=h^{0,1}(X_c)=0$ and $h^{2,0}(X_c)=h^{0,2}(X_c)=g$, it follows that $h^{1,1}(X_c)=10(g+1)=5\cdot3^c+5$. Therefore $r=0$ and $\rho(X_c)=h^{1,1}(X_c)$.
\end{proof}


\section{The $j$-Invariant of $f\colon X_c\rightarrow \mathbb{P}^1$}\label{S8}
For an elliptic fibration $\varphi\colon S\rightarrow C$ without multiple fibers, consider the rational map $j\colon C\dashrightarrow \mathbb{P}^1$ given by sending each point $P\in C$ such that $\varphi^{-1}(P)$ is nonsingular to the $j$-invariant of the elliptic curve $\varphi^{-1}(P)$. This rational map $j$ can in fact be extended to all of $C$ (see for instance \cite{kodaira}). The morphism $j\colon C \rightarrow \mathbb{P}^1$ is called the $j$-\emph{invariant} of the elliptic surface $\varphi
\colon S\rightarrow C$. 

If $P\in C$ is such that $f^{-1}(P)$ is singular, then we have the following (reproduced from \cite{kloosterman}):
\begin{center}
\begin{tabular}{|c|c|}
\hline
Fiber Type over $P$& $j(P)$\\
\hline
$I_0^*$ & $\ne \infty$\\
$I_b$, $I_b*$ $(b>0)$ & $\infty$\\
$II$, $IV$, $IV^*$, $II^*$& $0$\\
$III$, $III^*$ & $1728$\\
\hline
\end{tabular}
\end{center}

\begin{lem}\label{j-inv}
For $c\ge 2$, the $j$-invariant $j\colon \mathbb{P}^1\rightarrow \mathbb{P}^1$ of $f\colon X_c\rightarrow \mathbb{P}^1$ is non-constant.
\end{lem}
\begin{proof}
From Theorem \ref{sing}, all of the singular fibers of $f\colon X_c\rightarrow \mathbb{P}^1$ are of type $I_b$ or $I_b^*$ with $b>0$. Hence the $j$-invariant of $f\colon X_c\rightarrow \mathbb{P}^1$ satisfies $j(P)=\infty$ for all $P\in \mathbb{P}^1$ such that $f^{-1}(P)$ is singular. However, since generically for $P\in \mathbb{P}^1$ the $j$-invariant $j(P)$ is the $j$-invariant of the elliptic curve $f^{-1}(P)$, generically $j$ cannot be $\infty$. Thus $j$ is non-constant.
\end{proof}

\begin{prop}\label{j-inv2} For $c\ge 2$, the $j$-invariant $j:\mathbb{P}^1\rightarrow \mathbb{P}^1$ of $f:X_c\rightarrow \mathbb{P}^1$ has degree $6\cdot 3^c$ and is ramified at the points $0$, $1728$, and $\infty$. There are $2\cdot 3^c$ branch points above $0$, all of ramification index $3$. There are $3\cdot 3^{c}$ branch points above $1728$, all of ramification index $2$. Finally, there are $2$ branch points above $\infty$, one with ramification index $4\cdot 3^c$ corresponding to the point $0\in \mathbb{P}^1$ and one with ramification index $3^c$ corresponding to the point $\infty \in \mathbb{P}^1$.
\end{prop}

\begin{proof}
This follows directly from results of Mangala Nori in \cite{nori}.  In particular, Nori proves in \cite[Theorem 3.1]{nori} that an elliptic fibration $S\rightarrow B$ with non-constant $j$-invariant is extremal if and only if the fibration has no singular fibers of type $I_0^*$, $II$, $III$, or $IV$ and its $j$-invariant is ramified only over $0$, $1728$, and $\infty$ with ramification index $e_v$ for $v\in B$ satisfying $e_v=1$, $2$, or $3$ if $j(v)=0$ and $e_v=1$ or $2$ if $j(v)=1$.

We know from Corollary \ref{extremal} that $f\colon X_c\rightarrow \mathbb{P}^1$ is extremal and from Lemma \ref{j-inv} that it has non-constant $j$-invariant. Hence, it follows from  \cite[Theorem 3.1]{nori} that $j:\mathbb{P}^1\rightarrow \mathbb{P}^1$ is ramified only over the points $0$, $1728$, and $\infty$ and that
\[\mathrm{deg}(j)=\sum_{I_b} b + \sum_{I_b^*}b,\]
where the two sums occur over all the singular fibers of $f$ of type $I_b$ and of type $I_b^*$ respectively. 

From Theorem \ref{sing}, the fibration $f\colon X_c\rightarrow \mathbb{P}^1$ has one fiber of type $I_{4\cdot 3^c}$, a total of $3^c$ fibers of type $I_1$, and one fiber of type $I_{3^c}^*$. Thus $\mathrm{deg}(j)=6\cdot 3^c.$

Now let
\begin{equation*}
\begin{aligned}
\mathcal{R}_0&=\{v\in \mathbb{P}^1\mid j(v)=0\}\\
\mathcal{R}_{1728}&=\{v\in \mathbb{P}^1\mid j(v)=1728\}.
\end{aligned}
\end{equation*}

If $e_v$ denotes the ramification index of a point $v\in \mathbb{P}^1$, let
\begin{equation*}
\begin{aligned}
R_0&=\sum_{v\in \mathcal{R}_0}(e_v-1)\\
R_{1728}&=\sum_{v\in \mathcal{R}_{1728}}(e_v-1)
\end{aligned}
\end{equation*}

Then since Theorem \ref{sing} implies that $f$ has no singular fibers of type $II,$ $II^*,$ $III$, $III*$, $IV$, or $IV^*,$ Nori's calculations in the proof of
\cite[Lemma 3.2]{nori} yield the following three equations:
\begin{equation}R_0+R_{1728}=\frac{7\cdot\mathrm{deg}(j)}{6}\end{equation}
\begin{equation}\label{r0}R_0-\frac{2\cdot\mathrm{deg}(j)}{3}\ge0 \end{equation}
\begin{equation}\label{r1}R_{1728}-\frac{\mathrm{deg}(j)}{2}\ge0 \end{equation}
Observe that
$\frac{2\cdot\mathrm{deg}(j)}{3}+ \frac{\mathrm{deg}(j)}{2}=\frac{7\cdot\mathrm{deg}(j)}{6}.$
Therefore we must have equality in Equations (\ref{r0}) and (\ref{r1}). It follows that
$$R_0=\frac{2\cdot\mathrm{deg}(j)}{3}=4\cdot 3^c$$
$$R_{1728}=\frac{\mathrm{deg}(j)}{2}=3\cdot3^{c}.$$
Moreover, because equality holds in  (\ref{r0}), Nori's proof in \cite[Lemma 3.2]{nori} implies that 
$\mathrm{deg}(j)=3|\mathcal{R}_0|.$
Hence we have
\begin{equation}\label{calr0}|\mathcal{R}_0|=2\cdot 3^c.\end{equation}
Now from \cite[Theorem 3.1]{nori}, for any $v\in \mathcal{R}_0$, we must have $e_v\le 3$. Hence using (\ref{calr0}), it follows that 
$R_0\le 4\cdot 3^c.$
But we have already shown that in fact equality holds, therefore we have $e_v=3$ for all $v\in \mathcal{R}_0$.

Since $X_c$ is extremal and $f:X_c\rightarrow \mathbb{P}^1$ has no singular fibers of type $III^*$, Nori's results \cite[Theorem 3.1]{nori} also imply that $e_v=2$ for all $v\in \mathcal{R}_{1728}$.

Finally, we know $j$ has a pole of order $b_i$ at points $v_i\in \mathbb{P}^1$ where the fiber over $v_i$ is of type $I_{b_i}$ or of type $I_{b_i}^*$. Hence the result follows from Theorem \ref{sing}.
\end{proof}


\section{The Surface $X_c$ is Elliptic Modular}\label{S9}
We begin by giving a brief introduction to elliptic modular surfaces as defined by Shioda \cite{modsur}. 
\subsection{Preliminaries on Elliptic Modular Surfaces} 
Following Nori \cite{nori}, for an elliptic surface $\varphi\colon S\rightarrow C$ with $j$-invariant $j\colon C\rightarrow \mathbb{P}^1$, let us define
\[C'=C\backslash j^{-1}\{0,1728,\infty\}.\]
In particular, for every $v\in C'$, the fiber $F_v=\varphi^{-1}(v)$ is smooth. The sheaf $G=R^1\varphi_*\mathbb{Z}$ on $C$ is the \emph{homological invariant} of the elliptic surface $S$. The restriction of $G$ to $C'$ is then a locally constant sheaf of rank two $\mathbb{Z}$-modules. Consider the monodromy homomorphism $\rho\colon \pi_1(C')\rightarrow SL(2,\mathbb{Z})$ associated to $\varphi\colon S\rightarrow C$. Observe that $\rho$ both determines and is determined by the sheaf $G$.

Conversely, let $j\colon C\rightarrow \mathbb{P}^1$ be a holomorphic map from an algebraic curve $C$ to $\mathbb{P}^1$ and let $C'=C\backslash j^{-1}\{0,1728,\infty\}.$ Let $\mathcal{H}=\{z\in \mathbb{C} \mid \mathrm{Im} (z)>0\}$ be the upper half-plane in $\mathbb{C}$ and consider the elliptic modular function $J\colon \mathcal{H}\rightarrow \mathbb{P}^1\backslash\{0,1728,\infty\}$. Finally let $U'$ be the universal cover of $C'$. Then there exists a holomorphic map $w\colon U'\rightarrow \mathcal{H}$ such that the following diagram commutes:
\begin{equation}\label{mono}
\begin{tikzcd}
U'\arrow{d}{\pi} \arrow{r}{w}&\mathcal{H}\arrow{d}{J}\\
C'\arrow{r}{j} &\mathbb{P}^1\backslash\{0,1728,\infty\}.
\end{tikzcd}
\end{equation}
This map $w$ thus induces a homomorphism $\overline{\rho}\colon \pi_1(C')\rightarrow PSL(2,\mathbb{Z})$.

Now suppose $\rho\colon \pi_1(C')\rightarrow SL(2,\mathbb{Z})$ is a homomorphism making the following diagram commute:
\[
\begin{tikzcd}[column sep=0pt]
\pi_1(C')\arrow{rd}{\overline{\rho}} \arrow{rr}{\rho}&[-6]&SL(2,\mathbb{Z})\arrow{ld}\\
&PSL(2,\mathbb{Z})&.
\end{tikzcd}
\]
Then it is possible to construct a unique elliptic surface $\varphi\colon S\rightarrow C$ having $j$-invariant given by the holomorphic map $j\colon C\rightarrow \mathbb{P}^1$ and having homological invariant given by the sheaf $G$ associated to the homomorphism $\rho$ \cite[Section 8]{kodaira}.

So now consider any finite-index subgroup $\Gamma$ of the modular group $SL(2,\mathbb{Z})$ not containing $-\mathrm{Id}$. Then $\Gamma$ acts on the upper half plane $\mathcal{H}$ and the quotient $\Gamma \backslash \mathcal{H}$, together with a finite number of cusps, forms an algebraic curve $C_\Gamma$. For any other such subgroup $\Gamma'$, if $\Gamma\subset \Gamma'$, then the canonical map $\Gamma \backslash \mathcal{H}\rightarrow \Gamma' \backslash \mathcal{H}$ extends to a holomorphic map $C_\Gamma \rightarrow C_{\Gamma'}$. In particular, taking $\Gamma'=SL(2,\mathbb{Z})$ and identifying $C_{\Gamma'}$ with $\mathbb{P}^1$ via the elliptic modular function $J$, we get a holomorphic map
\[j_{\Gamma}\colon C_{\Gamma}\rightarrow \mathbb{P}^1.\]
Hence, as discussed, there exists a $w\colon U'\rightarrow \mathcal{H}$ fitting into a diagram (\ref{mono}) which induces a representation $\overline{\rho}\colon \pi_1(C')\rightarrow \overline{\Gamma}\subset PSL(2,\mathbb{Z})$, where $\overline{\Gamma}$ is the image of $\Gamma$ in $PSL(2,\mathbb{Z})$. Because $\Gamma$ contains no element of order $2$, this homomorphism $\overline{\rho}$ lifts to a homomorphism $\rho\colon \pi_1(C')\rightarrow SL(2,\mathbb{Z})$, which then gives rise to a sheaf $G_\Gamma$ on $C_\Gamma$.

\begin{defn}\cite{modsur} For any finite index subgroup $\Gamma$ of $SL(2,\mathbb{Z})$ not containing $-\mathrm{Id}$, the associated elliptic surface $\varphi\colon S_\Gamma\rightarrow C_\Gamma$ having $j$-invariant $j_\Gamma$ and homological invariant $G_\Gamma$ is called the \emph{elliptic modular surface} attached to $\Gamma$.  
\end{defn}

\subsection{The case of $X_c$}

We now return to considering the elliptic surface $f\colon X_c\rightarrow \mathbb{P}^1$. Let us define the following elements $A_0, A_1 \ldots, A_{3^c}, A_\infty$ of $SL(2,\mathbb{Z})$ as elements of the following conjugacy classes:
\[
A_0\in \left[\left(
\begin{array}{cc}
1&4\cdot 3^c\\
0& 1
\end{array}\right)\right]
\ \ \ \ \ \ \ 
A_1,\ldots, A_{3^c}\in \left[\left(
\begin{array}{cc}
1&1\\
0& 1
\end{array}\right)\right]
\ \ \ \ \ \ \ 
A_\infty\in \left[\left(
\begin{array}{cc}
-1&-3^c\\
0& -1
\end{array}
\right)
\right]
\]

Then consider the subgroup $\Gamma_c$ of index $6\cdot 3^c$ in $SL(2,\mathbb{Z})$ with the following presentation:
\[\Gamma_c:=\langle A_0, A_1,\ldots, A_{3^c}, A_{\infty} \mid A_0A_1\cdots A_{3^c}A_\infty=\mathrm{Id}\rangle.\]
We remark that $\Gamma_c$ is not a congruence subgroup as it does not appear on the list in \cite{experimental} of the genus $0$ congruence subgroups of $SL(2,\mathbb{Z})$ (see \cite{rademacher} for more details on such subgroups).

\begin{thm}\label{elmod} For $c\ge 2$, the surface $X_c$ is the elliptic modular surface attached to $\Gamma_c$. \end{thm}

\begin{proof}
In \cite[Theorem 3.5]{nori}, Mangala Nori proves that  extremal elliptic surface $\varphi\colon S\rightarrow C$ with a section and with non-constant $j$-invariant is an elliptic modular surface as long as $\varphi\colon S\rightarrow C$ has no singular fibers of type $II^*$ or $III^*$ in Kodaira's classification. Therefore, since the surface $f\colon X_c\rightarrow \mathbb{P}^1$ is extremal  (by Corollary \ref{extremal}), has a section (by Theorem \ref{sing}), has non-constant $j$-invariant (by Lemma \ref{j-inv}), and only has fibers of type $I_b$ and $I_b^*$ (by Theorem \ref{sing}), we know $X_c$  is indeed an elliptic modular surface. 

So let $\Gamma$ be the finite-index subgroup of $SL(2,\mathbb{Z})$ attached to $X_c$. By Proposition \ref{j-inv2}, the degree of the $j$-invariant of $f:X_c\rightarrow \mathbb{P}^1$ is $6\cdot 3^c$. Hence the group $\Gamma$ has index $6\cdot 3^c$ in $SL(2,\mathbb{Z})$. 

Now consider the $j$-invariant $j:\mathbb{P}^1\rightarrow \mathbb{P}^1$ of $X_c$, which we have investigated in Proposition \ref{j-inv2}. Let $C'=\mathbb{P}^1\backslash \{0,1728,\infty\}$.

Because $X_c$ is elliptic modular, its $j$-invariant induces a homomorphism 
\[\rho\colon \pi_1(C')\rightarrow \Gamma\subset SL(2,\mathbb{Z}).\]
Let us write the set
\[j^{-1}\{0,1728,\infty\}=\{v_1,\ldots,v_s\}.\]
By Proposition  \ref{j-inv2}, we know $s=5\cdot 3^c +2$. For each point $v_i$ let $\alpha_i$ be the loop element in $\pi_1(C')$ going around $v_i$. 
Then $\pi_1(C')$ is the free group on these generators $\alpha_1,\ldots, \alpha_s$ subject to the relation (taken in cyclic order) $\alpha_1\cdots\alpha_s=1$ \cite[Lemma 2.1]{bogomolov}.

In \cite[Proposition 1.4]{nori}, Nori describes, for an elliptic surface $S\rightarrow C$ with loop elements $\alpha_i\in \pi_1(C)$ around $v_i\in C$, the possible values of $\rho(\alpha_i)$ depending on the values of $j(v_i)$. 
By Proposition \ref{j-inv2}, for our elliptic surface $f\colon X_c\rightarrow \mathbb{P}^1$ all of the points $v_i$ such that $j(v_i)=0$ have ramification index $3$. Hence by \cite[Proposition 1.4]{nori}, for the corresponding $\alpha_i$, we have $\rho(\alpha_i)= \pm \mathrm{Id}$. However since $X_c$ is elliptic modular, the subgroup $\Gamma$ cannot contain $-\mathrm{Id}$. Hence, for all $i$ such that $j(v_i)=0$, we must have $\rho(\alpha_i) =\mathrm{Id}$. 

Similarly, by Proposition \ref{j-inv2} all of the points $v_i$ such that $j(v_i)=1728$ have ramification index $2$. But then by \cite[Proposition 1.4]{nori}, for all such $i$, we have $\rho(\alpha_i)=\pm \mathrm{Id}$ and thus, in fact, $\rho(\alpha_i)=  \mathrm{Id}$.

Therefore the only points $v_i\in j^{-1}\{0,1728,\infty\}$ that contribute non-identity elements to $\Gamma$ are the points sent to $\infty$ by $j$. These are exactly the points of $\mathbb{P}^1$ underneath the singular fibers of $f\colon X_c\rightarrow \mathbb{P}^1$. From \cite[Proposition 4.2]{modsur}, if a point $v_i$ has singular fiber of type $I_b$ with $b>0$, then
\[\rho(\alpha_i) \in \left[ \left( \begin{array}{cc} 1&b\\ 0 & 1\end{array}\right)\right].\]
If a point $v_i$ has singular fiber of type $I_b^*$ with $b>0$, then 
\[\rho(\alpha_i) \in \left[ \left( \begin{array}{cc} -1&-b\\ 0 & -1\end{array}\right)\right].\]
Therefore, using Theorem \ref{sing}, in the case of $f\colon X_c\rightarrow \mathbb{P}^1$, the point $0$ contributes a generator $A_0$ of $\Gamma$ in the conjugacy class of 
\[\left( \begin{array}{cc} 1&4\cdot 3^c\\ 0 & 1\end{array}\right)\]
in $SL(2,\mathbb{Z})$. Each point $\zeta^i$, for $\zeta$ a $3^c$-th root of unity, contributes a generator $A_{i+1}$ in the conjugacy class of 
\[\left( \begin{array}{cc} 1&1\\ 0 & 1\end{array}\right).\]
Finally, the point $\infty$ contributes a generator $A_\infty$ in the conjugacy class of
\[\left( \begin{array}{cc} -1&-3^c\\ 0 & -1\end{array}\right).\]
 
Then $\Gamma$ is the free group on these generators $A_0, A_1, \ldots, A_{3^c}, A_\infty$ subject to the relation $$A_0 A_1 \cdots A_{3^c}A_\infty=\mathrm{Id}.$$

Hence we indeed have that $\Gamma$ is the group $\Gamma_c$ defined above. 
\end{proof}

\textbf{Acknowledgements.}  The author would like to thank Matt Kerr, Jaclyn Lang, Christopher Lyons, Stefan Schreieder, and Burt Totaro for helpful discussions in the preparation of this article. The author gratefully acknowledges support of the National Science Foundation through awards DGE-1144087 and DMS-1645877.


\normalsize{\bibliography{EllipticModular.bib}}
\bibliographystyle{alpha}

\end{document}